\documentclass[a4paper,12pt]{amsart}
\usepackage[utf8]{inputenc}

\usepackage{amsmath}
\usepackage{amsfonts}
\usepackage{amssymb}
\usepackage{enumitem}

\newtheorem{Theorem}{Theorem}
\newtheorem{Corollary}[Theorem]{Corollary}

\newtheorem{lemma}{Lemma}[section]
\newtheorem{proposition}{Proposition}[section]
\newtheorem*{remark}{Remark}
\newcommand{\sph}{{\mathbb{S}^2}}

\newcommand{\G}{\operatorname{SO}(3)}
\DeclareMathOperator{\SO}{SO}

\newcommand{\ds}{d_{\sph}}
\newcommand{\IQ}{\mathbb{Q}}
\newcommand{\IC}{\mathbb{C}}
\newcommand{\la}{\langle}
\newcommand{\ra}{\rangle}
\providecommand{\norm}[1]{\lVert#1\rVert}
\providecommand{\abs}[1]{\lvert#1\rvert}
\def\N{\mathbb{N}}
\def\C{\mathbb{C}}

\title[Quantum Ergodicity and Averaging Operators on $\sph$]{Quantum Ergodicity and Averaging Operators on the Sphere}

\author{Shimon Brooks}
\address{Department of Mathematics, Bar-Ilan University, Ramat-Gan, 5290002 Israel}
\email{brookss@math.biu.ac.il}

\author{Etienne Le~Masson}
\address{Einstein Inst. of Math., The Hebrew University of Jerusalem, Jerusalem, 91904, Israel}
\email{lemasson@math.huji.ac.il}

\author{Elon Lindenstrauss}
\address{Einstein Inst. of Math., The Hebrew University of Jerusalem, Jerusalem, 91904, Israel}
\email{elon@math.huji.ac.il}

\begin{document}

\begin{abstract}

We prove quantum ergodicity for certain orthonormal bases of $L^2(\mathbb{S}^2)$, consisting of joint eigenfunctions of the Laplacian on $\mathbb{S}^2$ and the discrete averaging operator over a finite set of rotations, generating a free group. If in addition the rotations are algebraic we give a quantified version of this result. The methods used also give a new, simplified proof of quantum ergodicity for large regular graphs. 

\end{abstract}

\thanks{E.~Lindenstrauss\ and E.~Le~Masson\ were supported by ERC AdG Grant 267259. S.~Brooks was supported by NSF grant DMS-1101596, ISF grant 1119/13, and a Marie-Curie Career Integration Grant.  This paper was finalized while both E.~Lindenstrauss and E.~Le~Masson were in residence at MSRI, supported in part by NSF grant 0932078-000. E.~Lindenstrauss also thanks the Miller Institute for Basic Research in Science at UC  Berkeley and the Israeli IAS for their hospitality and support}
\maketitle

\section{Introduction}

Let $M$ be a compact Riemannian manifold, and $\Delta$ its Laplace-Beltrami operator.  We consider a basis of $L^2(M)$ consisting of $\Delta$-eigenfunctions; if there are multiplicities in the spectrum there may be many such.  The property of quantum ergodicity is a spectral analog of the classical ergodic property, stating that 
$$\frac{1}{N(\lambda)} \sum_{\lambda_j\leq \lambda} \left| \langle \psi_j, a\psi_j\rangle
- \int_M a\, d\text{Vol} \right|^2 \to 0$$
where $\{\psi_j\}$ runs over our basis of eigenfunctions, and $a$ is any fixed test function on $M$.  Here $N(\lambda):=|\{\lambda_j\leq \lambda\}|$ is the eigenvalue counting function (with multiplicity).  More generally, one replaces multiplication by $a$ in the inner product by a zero-order pseudodifferential operator $A$, and the $a$ in the integral by the principal symbol $\sigma_A$ of $A$.
The quantum ergodicity property implies that there exists a density $1$ subsequence $\{\psi_{j_k}\}$ of the $\{\psi_j\}$, such that  the measures $|\psi_{j_k}|^2dVol$ converge weak-* to the uniform measure.  This property was first shown to hold --- for any choice of orthonormal basis--- on manifolds $M$ for which the geodesic flow is ergodic, by \v{S}nirlman, Zelditch, and Colin de Verdi\`ere \cite{Shn,Z2,CdV}.  

It is important to note that quantum ergodicity is a property of the {\em basis}, rather than of the manifold $M$; if there are large degeneracies in the spectrum of $\Delta$, it may hold for some bases but not for others. Such large degeneracy occur for the spectrum of the Laplacian on the sphere $\mathbb S ^2$ --- indeed, for the sphere we can write
\begin{equation*}
L ^2 (\mathbb S ^2) = \bigoplus_ {s = 0} ^ \infty \mathcal H _ s
\end{equation*}
where for each $s \in \N$ we take $\mathcal H _ s$ to be the space of spherical harmonics of degree $s$, an eigenspace for the Laplacian with eigenvalue $s(s+1)$ and $\dim (\mathcal H _ s) = 2 s +1$; in particular $\mathcal H_0$ denotes the 1-dimensional space of constant functions.
Here it is easy to see that the standard basis of spherical harmonics fails to be quantum ergodic \cite{CdV}, but Zelditch \cite{ZelSphere} has shown that a random orthonormal basis of eigenfunctions will be quantum ergodic with probablity $1$ (in a natural sense). Strenghening this, VanderKam \cite{Vanderkam} has shown that quantum {\em unique} ergodicity holds for a random orthonormal basis, i.e. that if for any $s$ we chose $\{\psi^{(s)}_j\}_{j=1}^{2s+1}$ to be a random orthonormal basis of $\mathcal H_s$ then a.s. 
\begin{equation}\label{e:que}
\max_{1\leq j\leq 2s+1} \left| \langle \psi^{(s)}_j, a\psi^{(s)}_j\rangle
- \int_{\sph} a\, d\sigma \right| \to 0 \qquad\text{as $s\to\infty$}
\end{equation}
with $\sigma$ denoting the normalized volume measure on $\sph$.

\smallskip

Here, we consider bases of $L^2(\mathbb{S}^2)$, that consist of joint eigenfunctions of $\Delta$ and an averaging operator over rotations of the sphere. Precisely, 
for $N \geq 2$, let $g_1,\ldots,g_N$ be a finite set of rotations in $\G$, and define an operator $T_q$ acting on $L^2(\sph)$, by
\begin{equation}\label{Tq}
T_q f(x) = \frac{1}{\sqrt q}\sum_{j=1}^N (f(g_j x) + f(g_j^{-1} x)) 
\end{equation}
with  $q = 2N -1$; up to the choice of normalization constant this is an averaging operator. 
Since each space $\mathcal H_s$ is invariant under rotations, the self-adjoint operators $T _ q$ preserve $\mathcal H _ s$, and so we can find for every $s$ an orthonormal basis $\{\psi^{(s)}_j\}_{j=1}^{2s+1}$ of $T_q$-functions for $\mathcal H _ s$; thus  $\{\psi^{(s)}_j\}_{j,s}$ gives a complete orthonormal sequence of joint eigenfunctions of~$\Delta$ and~$T _ q$.
We note that for very special choice of rotations --- rotations that correspond to norm $n$ elements in an order in a quaternion division algebra --- it has been conjectured by B{\"o}cherer, Sarnak, and Schulze-Pillot \cite{B-S-SP} that such joint eigenfunctions satisfy the much stronger quantum unique ergodicity condition \eqref{e:que}. This conjecture is still open, but would follow from GRH---  or even from subconvex estimates on values of appropriate $L$-functions at the critical point (cf.~\cite{B-S-SP} for more details).

\smallskip
Our purpose in this paper is to prove quantum ergodicity for joint eigenfunctions of the Laplacian and a fairly general averaging operator on~$\sph$.
We give two flavours of such a result for these joint eigenfunctions: one in which the assumption on the rotations is milder --- namely that they generate a free subgroup of $\G$--- but allows only continuous test functions $a$; and one where $a$ can be rather wild --- an arbitrary element of $L^\infty(\mathbb S^2)$--- but the rotations are in addition assumed to be given by matrices whose entries are algebraic numbers. In this latter case we give \emph{quantitative} estimates. 

\begin{Theorem}\label{t:spheremain}
Let $g _ 1, \dots, g _ N$ be a finite set of rotations in $\G$ that generate a free subgroup. 
Let $\{ \psi^{(s)}_j \}_{j=1}^{2s+1}$ be an orthonormal basis of~$\mathcal{H}_s$ consisting of $T_q$-eigenfunctions, with~$T_q$ the operator defined in \eqref{Tq}.
Then for any continuous function~$a$ on~$ \sph$, we have
\[ \frac 1{2s+1} \sum_{j=1}^{2s+1} \left| \la \psi_j^{(s)}, a \psi_j^{(s)} \ra - \int_{\sph} a \, d \sigma \right|^2 \to 0 \]
when $s \to \infty$.
\end{Theorem}

It is not hard to show in this case that as $s\to\infty$, the distribution of the $T_q$-eigenvalues of $\{\psi ^ {( s )} _ j\}_{j=1}^{2s+1}$ tends to a specific measure, the Kesten-McKay distribution, whose support is the interval $[-2, 2]$ (see \S\ref{s:Kesten-McKay} for details), which implies the following version of Theorem~\ref{t:spheremain}.
\begin{Corollary}\label{c:spheremain}
Let $g _ 1, \dots, g _ N$ and $T _ q$ be as in Theorem~\ref{t:spheremain}, and let $I$ be an arbitrary (fixed) subinterval of $[-2,2]$. Let as before $\{ \psi^{(s)}_j \}_{j=1}^{2s+1}$ be an orthonormal basis of $\mathcal{H}_s$ consisting of $T_q$-eigenfunctions, and define $\lambda(s,j)$ by $T_q \psi^{(s)}_j = \lambda(s,j) \psi^{(s)}_j$, and
$$N(I,s)= \# \{ j : \lambda(s,j) \in I \}.$$
Then for any continuous function~$a$ on~$ \sph$, we have
\[ \frac 1{N(I,s)} \sum_{\{ j : \lambda(s,j) \in I \} } \left| \la \psi_j^{(s)}, a \psi_j^{(s)} \ra - \int_{\sph} a \, d \sigma \right|^2 \to 0 \]
when $s \to +\infty$.
\end{Corollary}

Finally, we give the following more quantitative result under more stringent assumptions on the rotations:

\begin{Theorem}\label{t:spherealgebraic}
Let $g _ 1, \dots, g _ N$ be a finite set of rotations in $\G$, that generate a free subgroup, and moreover so that all entries appearing in the matrices representing the $g _ i$ are given by algebraic numbers.
Let $\{ \psi^{(s)}_j \}_{j=1}^{2s+1}$ be an orthonormal basis of~$\mathcal{H}_s$ consisting of $T_q$-eigenfunctions with~$T_q$ defined from the $g _ i$ as above.
Then for any function $a \in L ^ \infty (\sph)$, we have
\[ \frac 1{2s+1} \sum_{j=1}^{2s+1} \left| \la \psi_j^{(s)}, a \psi_j^{(s)} \ra - \int_{\sph} a \, d \sigma \right|^2 \lesssim \frac {\norm {a} _ \infty }{ \log s},
\]
with the implicit constant depending only on $q=2N-1$ and the choice of rotations $g _ 1, \dots, g _ N$.
\end{Theorem}

A key ingredient in this sharper variant is a result of Bourgain and Gamburd \cite{BG} which shows that for such rotations the action of the operator $T_q$ on $L^2( \mathbb S^2)$ has a spectral gap.

These results are closely related to the results of the second-named author with N. Anantharaman \cite{ALM}, proving quantum ergodicity for expander sequences of large regular graphs that satisfy a non-degeneracy condition, that the radius of injectivity grow at almost all points (see condition (BST) below).
In Theorem~\ref{t:spherealgebraic}, the spectral gap result of Bourgain and Gamburd is a close analog to the expander condition on the graphs. The algebraicity of the rotations is also used to give a lower bound on how close a nontrivial word of length $n$ in $g _ 1$,\dots,$g_N$ can be, which can be viewed as an analogue of the injectivity radius condition. In Theorem~\ref{t:spheremain}, we reduce to the case of $a$ being a sum of finitely many spherical harmonics, where we get a spectral gap for $T_q$ simply because $a$ belongs to a finite dimensional $T_q$-invariant subspace of $L^2(\sph)$, and the assumption that $g _ 1, \dots, g _ N$ generate a free group gives a weaker, unquantified, substitute for the injectivity radius condition.

We will use a new and alternative method to \cite{ALM}, more straightforward in the sense that it does not use pseudo-differential calculus on the graphs (\cite{LM}), but instead a discrete wave propagation.  Although it gives a less general result in the present form, it may be of independent interest, and gives a quantitative improvement in some aspects. We thus bring in Sections~\ref{qe on graphs} and \ref{main estimate graphs} a new proof of quantum ergodicity for large regular graphs (i.e., the main result of \cite{ALM}), in order to introduce the ideas to be used in the sequel.  Then in Section~\ref{sphere}, we return to the setting of $\mathbb{S}^2$, and show the modifications necessary to work on the sphere in $\mathcal{H}_s$, and prove Theorems~\ref{t:spheremain} and \ref{t:spherealgebraic}.

\section{Quantum ergodicity on regular graphs}\label{qe on graphs}

Let $(G_k)$ be a sequence of connected $(q+1)$-regular graphs, $G_k=(V_k, E_k)$ with $V_k =\{1,\ldots, k\}$.
We are interested in the eigenfunctions of the averaging operator
$$T_q f(x) = \frac1{\sqrt{q}} \sum_{d(x,y)=1} f(y) $$
acting on functions of the vertices of the graph.
We  parametrize the spectrum of $T_q$ by $\lambda = 2\cos{\theta_\lambda}$.
The spectrum of this operator is contained in $$\left[-2\cosh\left(\frac{\log q}{2}\right),2\cosh\left(\frac{\log q}{2}\right)\right].$$ It can be divided into two parts: the \emph{tempered spectrum} is the part contained in the interval $[-2,2]$, and the \emph{untempered spectrum} is the part lying outside this interval.
We assume the following conditions:

(EXP) The sequence of graphs is a family of expanders. That is, there exists a fixed $\beta>0$ such that the spectrum of $T_q$ on $L^2(G_k)$ is contained in $\left\{2\cosh\left(\frac{\log q}{2}\right)\right\}\cup \left[-2\cosh\left(\frac{\log q}{2} - \beta \right), 2\cosh\left(\frac{\log q}{2} - \beta \right)\right]$ for all $k$.

(BST) The sequence of graphs converges to a tree in the sense of Benjamini-Schramm. 
More precisely, for all $R$, $\frac{|\{x\in V_k, \rho(x)<R\}|}{k}\to 0$ where $\rho(x)$ is the injectivity radius at $x$ (meaning the largest $\rho$ such that the ball $B(x, \rho)$ in $G_k$ is a tree). Or equivalently, there exists $R_k\to +\infty$ and $\alpha_k\to 0$ such that $$\frac{|\{x\in V_k, \rho(x)<R_k\}|}{k}\leq \alpha_k.$$

The last condition is satisfied in particular if the injectivity radius goes to infinity (with $R_k$ taken to be the minimal injectivity radius and $\alpha_k=0$).

\begin{Theorem} \label{t:main}  Assume that $(G_k)$ satisfies (BST) and (EXP). Denote by $\{\psi^{(k)}_1,\ldots, \psi^{(k)}_k\}$ an orthonormal basis of eigenfunctions of $T_q$ on $G_k$.

Let $a_k : V_k \to \C$ be a sequence of functions such that 

$$\sum_{x\in V_k} a_k(x)=0 \quad \text{and} \quad   \norm{a_k}_\infty \leq 1.$$

Then $$\frac1{k} \sum_{j=1}^k \left| \la \psi^{(k)}_j, a_k \psi^{(k)}_j\ra\right|^2\to 0,$$
when $k\to +\infty$.

More precisely,
 \begin{align}\label{e:qegraph}
  \frac1{k} \sum_{j=1}^k \left| \la \psi^{(k)}_j, a_k \psi^{(k)}_j\ra\right|^2
   &\lesssim \beta^{-2} \min\left\{ R_k, \log(1/ \alpha_k) \right\}^{-1} \|a\|_2^2 + \alpha_k^{1/2} \|a\|_\infty^2,
 \end{align}
 where the implicit constant depends only on the degree $q+1$.
\end{Theorem}

\begin{remark}
The spectral average in \eqref{e:qegraph} can be restricted to eigenvalues in any fixed subinterval of $\left[-2\cosh\left(\frac{\log q}{2}\right),2\cosh\left(\frac{\log q}{2}\right)\right]$, using the Kesten-McKay law (see for example Section 5 of  \cite{ALM}).
\end{remark}

 We now introduce the Chebyshev polynomials $P_n$ of the first kind for every $n \in \N$ defined by the relation
$$ P_n(\cos\theta) = \cos(n\theta). $$
As each $P_n$ is a polynomial, the definition of the operator $P_{n}(T_q/2)$ is clear.

We denote by $\norm{A}_{HS}$ the Hilbert-Schmidt norm of the operator $A$ acting on $L^2(G_k)$, and write $L^2_0(G_k)$ for the subspace of $L^2(G_k)$ orthogonal to the constants.  We will prove the following proposition in Section~\ref{main estimate graphs}.

\begin{proposition}\label{p:egorov_graphs}
Let $a$ be a function in $L^2_0(G_k)$. For any integer $T \geq 1$,
 \begin{multline*}
 \left\| \frac1T \sum_{n=1}^T P_{2n}(T_q/2) a P_{2n}(T_q/2) \right\|_{HS}^2 
 \lesssim \frac{\|a\|_2^2}{\beta^2 T} \\
 + q^{8T} |\{x\in G_k, \rho_k(x) \leq 4T\}| \norm{a}_\infty^2,
 \end{multline*}
 where $\rho_k(x)$ is the radius of injectivity at the vertex $x \in G_k$, and the implicit constant in the inequality depends only on $q$.
\end{proposition}

To prove the theorem, we first need the following lemma, which will be used as well in the case of the sphere in Section~\ref{sphere}.
\begin{lemma}\label{l:unitarity}
 For every $\theta \in [0,\pi]$, $T \geq 10$,
 $$ \left| \frac1T \sum_{n=1}^T \cos(n\theta)^2 \right| \geq 0.3$$ 
\end{lemma}

\begin{proof}[Proof of Lemma~\ref{l:unitarity}] 
We write
$$ S_T := \sum_{n=1}^T \cos(n\theta)^2. $$
Then
\begin{align*}
\frac1T S_T &= \frac1T \sum_{n=1}^T \left( \frac12 + \frac14 e^{2in\theta} + \frac14 e^{-2in\theta} \right) \\
&= \frac{2T - 1}{4T} + \frac1{4T} \sum_{n=-T}^T e^{2in\theta} \\
&= \frac{2T - 1}{4T} + \frac{\sin[(2T+1)\theta]}{4T \sin\theta}.
\end{align*}
Now 
$$ \min_{\theta \in [0,\pi]} \frac{\sin[(2T+1)\theta]}{4T \sin\theta} = \min_{\theta \in [0,\pi/2]} \frac{\sin[(2T+1)\theta]}{4T \sin\theta} \geq -\frac1{4T\sin(\pi/(2T+1))}.$$
Indeed: $|\sin\theta|$ is monotone increasing in $[0,\pi/2]$ and $\sin(2T+1)\theta > 0$ for $0 < \theta < \pi/(2T+1).$

Hence $$ \frac1T S_T \geq \frac1{4T} \left( 2T-1 - \frac1{\sin(\pi/(2T+1))} \right).$$
A routine calculation shows this is $\geq 0.3$ for $T\geq 10$.
\end{proof}

We now show how Theorem~\ref{t:main} follows from the central estimate Proposition~\ref{p:egorov_graphs}; the proof of Proposition~\ref{p:egorov_graphs} will be the subject of Section~\ref{main estimate graphs}.
\begin{proof}[Proof of Theorem \ref{t:main} from Proposition~\ref{p:egorov_graphs}]
Recall that 
$$P_{2n}(T_q/2)\psi_j^{(k)} = \cos(2n\theta_{j,k})\psi_j^{(k)}$$
for $\lambda_j^{(k)}=2\cos\theta_{j,k}$.  
Thus we use Lemma~\ref{l:unitarity} to estimate
 \begin{align*}
  \frac1{k} &\sum_{j=1}^k \left| \la \psi^{(k)}_j, a_k\psi^{(k)}_j\ra\right|^2 \\
  & \lesssim  \frac1{k} \sum_{j=1}^k \left|\frac{1}{T} \sum_{n=1}^T \cos^2 (2n\theta_{j,k})\right|^2\left| \la \psi^{(k)}_j, a_k\psi^{(k)}_j\ra\right|^2 \\
  &\lesssim \frac1{k} \sum_{j=1}^k \left| \left\la \psi^{(k)}_j, \frac1T \sum_{n=1}^T P_{2n}(T_q/2) a_k P_{2n}(T_q/2)\psi^{(k)}_j\right\ra\right|^2 \\
  &\lesssim \frac1{k} \left\| \frac1T \sum_{n=1}^T P_{2n}(T_q/2) a_k P_{2n}(T_q/2) \right\|_{HS}^2 \\
  &\lesssim \frac{\|a\|_{2}^2}{k T} + q^{8T} \frac{|\{x\in G_k, \rho(x) \leq 4T\}|}{k} \|a\|_\infty^2,
 \end{align*}  
 with the implicit constant depending only on $q$, as in Proposition~\ref{p:egorov_graphs}.
We then take $T = T_k = \min\left\{ \frac{R_k}{4}, -\frac{\log \alpha_k}{16 \log q} \right\}$.
\end{proof}

\section{Proof of the main estimate for graphs}\label{main estimate graphs}

Let $\mathfrak{X}$ be the $q+1$ regular tree, and $G = \mathfrak{X}/\Gamma$ be any finite $q+1$-regular graph (where $\Gamma$ is a discrete subroup of the group of automorphisms of the tree). We will denote by $d(x,y)$ the geodesic distance on the tree. We will also assume that when the symbol $\lesssim$ is used without further indication, the implicit constant depends only on $q$.

Proposition~\ref{p:egorov_graphs} (and later Proposition~\ref{p:egorov}) is based on the following propagation lemma; see  \cite{BL} or \cite{jointQmodes} for a proof.
\begin{lemma}\label{estimate} 
Fix a point $0 \in \mathfrak{X}$ and write $|x| = d(0,x)$ for any $x \in \mathfrak{X}$.
Let $\delta_0$ be the $\delta$-function supported at $0$, and $n$ a positive even integer.  Then
\begin{eqnarray*}
P_n(T_q/2)\delta_0(x) & = & \left\{ \begin{array}{ccc} 0  & \quad & |x|  \text{ odd } \quad \text{or} \quad |x|>n\\ \frac{1-q}{2q^{n/2}} & \quad & |x|<n \quad \text{and} \quad |x| \text{ even } \\ \frac{1}{2q^{n/2}} & \quad & |x| = n \end{array}\right.
\end{eqnarray*}
\end{lemma}

In what follows, we will denote simply by $P_n$ the operator $P_n(T_q/2)$. The set $D$ will be a fundamental domain of $G$ on the tree $\mathfrak{X}$. We see any function $a$ on the graph $G$ as a $\Gamma$-invariant function on the tree.
We denote by $K_n(x,y)$ the kernel of the operator $P_n a P_n$ on the tree. This kernel satisfies the invariance relation
$$ \forall \gamma \in \Gamma \quad K_n(\gamma \cdot x, \gamma \cdot y) = K_n(x,y) $$
and it therefore defines an operator $T_n$ on the graph $G$, whose kernel is given by 
$$ \tilde{K}_n(x,y) = \sum_{\gamma \in \Gamma} K_n(x,\gamma \cdot y) $$
We denote by $K_T(x,y)$ the kernel of $A_T = \frac1T \sum_{n=0}^T P_{2n} a P_{2n}$, and by $\tilde{K}_T$ the corresponding kernel on the graph. Let us first give an explicit expression of the kernel $K_T$. For this purpose we define the sets
\[E_{j,k} = E_{j,k}(x,y) = \{ z: d(x,z) = j, d(y,z) = k \}.\] 
We find from Lemma~\ref{estimate} that
\begin{align}\label{e:K2n}
 K_{2n}(x,y) &= \frac1{4q^{2n}} \sum_{z\in E_{2n,2n}} a(z) \\
 &\quad + \frac{(1-q)}{4q^{2n}} \sum_{j=0}^{n-1}\left( \sum_{z\in E_{2j,2n}} a(z) + \sum_{z\in E_{2n,2j}} a(z) \right) \nonumber \\
 &\quad + \frac{(1-q)^2}{4q^{2n}} \sum_{j,k=0}^{n -1} \sum_{z\in E_{2j,2k}} a(z). \nonumber
\end{align}
Note that this kernel is equal to $0$ whenever $d(x,y)$ is odd or $d(x,y) > 4n$. We then have $K_T = \frac1T \sum_{n=0}^T K_{2n}$.

We want to evaluate the Hilbert-Schmidt norm of this operator on the graph, that is
$$ \norm{A_T}_{HS}^2 = \sum_{x,y \in D} |\tilde{K}_T(x,y)|^2. $$
We first separate the points with small and large radius of injectivity. Define
\begin{equation*}
A_T' f(x) = \left\{
\begin{array}{ll}
   A_T f(x) & \text{if } \rho(x) > 4T \\
   0 & \text{otherwise.}
\end{array}
\right.
\end{equation*}
We then have
\begin{lemma}
$$ \|A_T\|_{HS}^2 \leq \| A_T' \|_{HS}^2 + q^{8T}  |\{x\in D, \rho(x) \leq 4T\}| \| a \|_\infty^2 $$
\end{lemma}

\begin{proof}
We have
$$
\| A_T - A'_T\|_{HS}^2 = \frac1{T^2} \sum_{\substack{x,y\in D\\ \rho(x) \leq 4T}} \left|  \sum_{\gamma\in \Gamma} \sum_{n=0}^T K_{2n}(x,\gamma\cdot y) \right|^2 
$$
and there are at most $q^{4T}$ terms in the sum over $\Gamma$. So we have by Cauchy-Schwarz inequality, and using also the fact that $K_{2n}(x,y) = 0$ when $d(x,y) > 4T \geq 4n$, 
\begin{align*}
 \| A_T - A'_T\|_{HS}^2 
 &\leq \frac1{T^2} q^{4T} \sum_{\substack{x,y\in D\\ \rho(x) \leq 4T}}\sum_{\gamma \in \Gamma}  \left| \sum_{n=0}^T K_{2n}(x,y) \right|^2\\
 &\leq \frac1{T^2} q^{4T} \sum_{\substack{x\in D\\ \rho(x) \leq 4T}}\sum_{\substack{y \in \mathfrak{X} \\ d(x,y) \leq 4T}}  \left| \sum_{n=0}^T K_{2n}(x,y) \right|^2.
 \end{align*}
We then use the fact that $\sup_{x,y} K_{2n}(x,y) \leq \| a \|_\infty$ as is clear from \eqref{e:K2n} to obtain
$$ \| A_T - A'_T\|_{HS}^2 \leq q^{8T}  |\{x\in D, \rho(x) \leq 4T\}| \| a \|_\infty^2$$
\end{proof}

We can now restrict our attention to the operator $A_T'$. We write 
$$ A_T' = \frac1T \sum_{n=0}^T \frac1{q^{2n}} \sum_{l=0}^{2n}\sum_{j,k=0}^n c(j,k) \tilde{S}_{2j,2k,2l}, $$
or, interchanging the sums in $l$ and $n$, which will be useful later:
\begin{equation}\label{e:AT}
 A_T' = \frac1T \sum_{l=0}^{2T} \sum_{n=\lceil l/2 \rceil}^{T} \frac1{q^{2n}} \sum_{j,k=0}^n c(j,k) \tilde{S}_{2j,2k,2l},
\end{equation}
where $\tilde{S}_{2j,2k,2l}$ is defined as follows. We first define the operator $S_{2j,2k,2l}$, acting on functions of the tree, by its kernel
$$ [S_{2j,2k,2l}](x,y) = \mathbf{1}_{\{d(x,y)=2l \}} \sum_{z\in E_{2j,2k}(x,y)} a(z).  $$
It gives an operator on the graph, that we restrict to the points $x$ such that $\rho(x) > 4T$ in order to obtain $\tilde{S}_{2j,2k,2l}$: the kernel on the graph is, for $x,y \in D$
$$  [\tilde S_{2j,2k,2l}](x,y) = \mathbf{1}_{\{\rho(x) > 4T\}}  \sum_{\gamma \in \Gamma}  [S_{2j,2k,2l}](x,\gamma \cdot y). $$
 Note that the constants $c(j,k)$ depend only on $q$.

The proof of Proposition \ref{p:egorov_graphs} then follows from the following estimate:
 \begin{lemma}\label{l:Sop}
$$ \| \tilde{S}_{2j,2k,2l} \|_{HS} \lesssim q^{(k+j)} e^{-\frac{\beta}2 (k+j-l)} \|a \|_2 $$
 \end{lemma}
 Indeed, let us first notice that for $l\neq l'$, the operators $\tilde{S}_{2j,2k,2l}$ and $\tilde{S}_{2j',2k',2l'}$ are orthogonal with respect to the Hilbert-Schmidt norm. We deduce from this fact and the expression \eqref{e:AT} that
 \begin{align*}
\|A_T' \|^2_{HS} &= \frac1{T^2} \sum_{l=0}^{2T} \left\| \sum_{n=\lceil l/2 \rceil}^T \frac1{q^{2n}} \sum_{j,k=0}^n   c(j,k) \tilde S_{2j,2k,2l} \right\|_{HS}^2 \\
&\lesssim \frac1{T^2} \sum_{l=0}^{2T} \left( \sum_{n=\lceil l/2 \rceil}^T \frac1{q^{2n}} \sum_{j,k=0}^n \| \tilde{S}_{2j,2k,2l} \|_{HS}  \right)^2 \\
&\lesssim \frac1{T^2} \sum_{l=0}^{2T} \left( \sum_{n=\lceil l/2 \rceil}^T \frac1{q^{2n}} \sum_{j,k=0}^n q^{(k+j)} e^{-\frac{\beta}2 (k+j-l)} \|a\|_2 \right)^2 \\
&\lesssim \frac1{T^2} \sum_{l=0}^{2T}  \left( \sum_{n=\lceil l/2 \rceil}^T e^{-\beta(n-l/2)} \|a\|_2 \right)^2 \\
&\lesssim \frac{ \|a\|^2_2}{\beta^2 T}.
\end{align*}

Let us now prove Lemma \ref{l:Sop}. Recall that the kernel of $S_{2j,2k,2l}$ on the tree is given by
 $$ \mathbf{1}_{\{d(x,y)=2l \}} \sum_{z\in E_{2j,2k}(x,y)} a(z).  $$
In order to be able to work with this expression, we will consider the \emph{arc graph}, defined as follows. For every directed edge $a$ of $G$ we denote by $a^+$ the target of $a$, by $a^-$ the source of $a$ and by $\bar a$ the reversal of $a$. The arc graph $G'$ is then the directed graph whose vertices are the directed edges, or arcs of $G$, and such that there is an edge from the arc $a$ to the arc $b$ when $a^+ = b^-$ and $a \neq \bar b$. We will also see the arc graph $G'$ as a quotient of the arc tree $\mathfrak{X}'$ by the group $\Gamma$ acting on edges.

We then define the averaging operator (or normalized adjacency matrix) $T_q'$ on $G'$ for any function $F$ on $G'$ by
$$T_q' F(a) = \frac1q \sum_{\substack{b^- = a^+ \\ b \neq \bar a}} F(b).$$
We also define the maps $B,E : L^2(G) \to L^2(G')$  by
\begin{equation}\label{e:beginend}
 Bf(a) = f(a^-), Ef(a) = f(a^+) 
\end{equation}

\begin{lemma}\label{l:edgegap}
For every $k \geq 1$, we have
$$ \| (T_q')^k \| \lesssim e^{-\beta k}$$
with an implicit constant depending only on $q$.
\end{lemma}
\begin{proof}
The idea is to find an orthonormal basis in which the powers of $T_q'$ are simple to study. The space $L^2(G')$ can be divided into the direct sum of $\text{Im}(B)\oplus\text{Im}(E)$ and its orthogonal complement, the latter being the space of functions $F$ such that for any $v \in G$
$$ \sum_{a^- = v} F(a) = \sum_{a^+=v} F(a) = 0.$$
On this space the action of $T_q'$ is given by
$$ T_q'F(a) = - \frac{1}{q} F(\bar a). $$

We can then decompose $\text{Im}(B)\oplus\text{Im}(E)$ using an orthonormal basis of $T_q$. For every eigenfunction $w \in L^2(G)$ with $T_q$-eigenvalue $\lambda$, the space $\C(Bw) + \C(Ew)$ is stable under $T_q'$ and the action of $T_q'$ is given by the matrix
\begin{equation*}
 \begin{pmatrix}
  0 & -\frac{1}{q} \\
  1 & \frac{\lambda}{\sqrt{q}}
 \end{pmatrix}
\end{equation*}
Notice that we have the conjugation
\begin{equation*}
  \begin{pmatrix}
  q^{1/4} & 0 \\
  0 & q^{-1/4}
 \end{pmatrix}
  \begin{pmatrix}
  0 & -\frac{1}{q} \\
  1 & \frac{\lambda}{\sqrt{q}}
 \end{pmatrix}
  \begin{pmatrix}
  q^{-1/4} & 0 \\
  0 & q^{1/4}
 \end{pmatrix} = 
  \frac1{\sqrt{q}} \begin{pmatrix}
  0 & -1 \\
  1 & \lambda
 \end{pmatrix},
\end{equation*}
so, up to a constant depending only on $q$ it is enough to study the matrix on the right-hand side of the previous equality.
We have
\begin{equation}\label{e:chebyshevmatrix}
q^{-n/2}
 \begin{pmatrix}
  0 & -1 \\
  1 & \lambda
 \end{pmatrix}^n
 = q^{-n/2}
  \begin{pmatrix}
  -U_{n-2}(\lambda/2) & -U_{n-1}(\lambda/2) \\
  U_{n-1}(\lambda/2) & U_n(\lambda/2)
 \end{pmatrix}
\end{equation}
where the $U_n$ are the Chebyshev polynomials of the second kind, $$U_n(\cosh\theta) = \sinh[(n+1)\theta]/\sinh\theta.$$ Because $|\lambda| \leq 2\cosh\left(\frac{\log q}{2} -\beta \right)$ the coefficients of \eqref{e:chebyshevmatrix} are bounded in absolute value by $e^{-\beta n}$ up to a constant depending only on $q$.
\end{proof}

We can now rewrite the kernel of $S_{2j,2k,2l}$ using the operator $T'_q$ with the help of the following lemma.

  \begin{lemma}\label{l:Tq sum}
  Let $ 0 \leq j,k \leq n$.
 Let $x \in D$ and $y \in \mathfrak{X}$ such that $l = d(x,y)/2$ satisfies $|k-j| \leq l \leq k+j$. 
 Denote by $w$ the vertex of the segment $[x,y]$ such that $d(x,w) = l - (k-j)$ and $d(y,w) = l + (k-j)$ in the case $k\geq j$, or $d(x,w) = l + (k-j)$ and $d(y,w) = l - (k-j)$ in the case $j \geq k$. Then
  \begin{equation}\label{e:vertextoedge}
  \sum_{z\in E_{2j,2k}} a(z) = \sideset{}{'}\sum_{e^- = w} q^{k+j-l} (T'_q)^{k+j-l} B a (e)
  \end{equation}
  where the sum on the right-hand side runs over the edges which contain only one vertex on the segment $[x,y]$.
 \end{lemma}
We will also need the following lemma whose proof is clear.

\begin{lemma}\label{l:edgenorm}
 Let $k \leq 2l$ be fixed, and for any $x \in D$ and $y \in \mathfrak X$ such that $d(x,y) = 2l$, let $w$ be the vertex of the segment $[x,y]$ such that $d(x,w) = k$.
 $$ \sum_{\substack{x\in D\\ \rho(x) > 4T}} \sum_{\substack{y\in \mathfrak{X}\\ d(x,y)=2l}} \sum_{\substack{e\in\mathfrak{X}' \\ e^- = w}} |f (e)| \lesssim q^{2l} \sum_{e \in D'}  | f (e)| $$
  where $D'$ is a fundamental domain for the action of $\Gamma$ on $\mathfrak{X}'$.
\end{lemma}

\begin{proof}[Proof of Lemma \ref{l:Sop}]
We use Lemma \ref{l:Tq sum} and Lemma \ref{l:edgenorm} to write
 \begin{align*}
 \| \tilde S_{2j,2k,2l} \|_{HS}^2 & = \sum_{\substack{x\in D\\ \rho(x) > 4T}} \sum_{\substack{y\in \mathfrak X\\ d(x,y)=2l}} \left| \sideset{}{'}\sum_{e^- = w} q^{k+j-l} (T'_q)^{k+j-l} B a (e) \right|^2 \\
 &\lesssim q^{2(k+j-l)} q^{2l} \sum_{e \in D'} \left| (T'_q)^{k+j-l} B a (e) \right|^2.
 \end{align*}
 Then Lemma \ref{l:edgegap} gives
 $$ \| \tilde S_{2j,2k,2l} \|_{HS}^2 \lesssim q^{2(k+j)} e^{-\beta(k+j-l)} \|a\|_2^2.$$
\end{proof}

\section{The argument on the sphere}\label{sphere}

We bring here the additional ingredients needed to adapt the methods of the previous sections to the setting of $\sph$, in order to prove Theorems~\ref{t:spheremain} and \ref{t:spherealgebraic}.  We begin by reviewing some harmonic analysis on $\sph$ that we will need; further details can be found in eg. \cite{Hobson, Szego}.

\subsection{Some Harmonic Analysis on $\sph$}\label{s:harmonic analysis}

Eigenfunctions of the Laplacian on $\sph$ can be realized as restrictions to $\sph$ of harmonic polynomials in $\mathbb{R}^3$.  The Laplace eigenvalues are given by $s(s+1)$ for $s\in\mathbb{N}$, and the corresponding eigenspace $\mathcal{H}_s$ of spherical harmonics has dimension $2s+1$.  There is a unique $L^2$-normalized function in each $\mathcal{H}_s$ that is radial with respect to the North pole $z=0$. We call it $Y_s$ and it is given in $(\theta,\phi)$ coordinates by
$$Y_s(\theta, \phi)= Y_s(\theta) = \sqrt{\frac{2s+1}{4\pi}} L_s(\cos\theta)$$
where $L_s$ is the {\bf Legendre polynomial} of degree $s$.  In fact, an orthonormal basis of the eigenspace $\mathcal{H}_s$ is given by
$$Y_s^m(\theta, \phi)=(-1)^m\sqrt{\frac{2s+1}{4\pi}\frac{(s-m)!}{(s+m)!}}L_s^m(\cos\theta)e^{im\phi}$$
as $m$ runs over $-s\leq m\leq s$, where $L_s^m$ is the {\bf associated Legendre polynomial} of degree $s$ and order $m$.

We also define the {\bf zonal spherical harmonics} $Z^{(s)}_z$ to be the unique element of $\mathcal{H}_s$ that is radial with respect to $z\in\sph$, and is normalized by
$$Z_0^{(s)} (\theta, \phi) = \sqrt{\frac{2s+1}{4\pi}}Y_s (\theta, \phi) = \frac{2s+1}{4\pi} L_s(\cos\theta)$$
With this normalization, $Z_z^{(s)}$ has the reproducing property on $\mathcal{H}_s$
\begin{equation}\label{reproducing}
\langle Z_z^{(s)}, \psi\rangle = \psi(z)
\end{equation}
for any $\psi\in\mathcal{H}_s$.  The zonal spherical harmonics can  be expressed in terms of the $Y_s^m$--- indeed, in terms of any orthonormal basis of $\mathcal{H}_s$--- by
$$Z_z^{(s)}(y) = \sum_{m=-s}^s Y_s^m(z)\overline{Y_s^m(y)}$$
We will also make use of the estimate \cite[Theorem 7.3.3]{Szego}
\begin{equation}\label{legendre decay}
|L_s(\cos\theta)| < \frac{1}{\sqrt{s \sin\theta}} < \frac{2}{\sqrt{s \theta}}
\end{equation}
for $0\leq \theta\leq \frac{\pi}{2}$, which implies 
\begin{equation}\label{e:Zest} 
|Z_z^{(s)}(y)| \lesssim\sqrt{\frac{s}{d_\sph (z,y)}}
\end{equation}
for points $z,y$ in the same hemisphere.

We can use this to estimate the inner product between two zonal spherical harmonics:
\begin{lemma}\label{inner product z z'}
Let $z,z'\in\sph$.  Then
$$\left| \langle Z_z^{(s)}, Z_{z'}^{(s)}\rangle\right| < 2 \sqrt{\frac{s}{{\ds^{\pm}(z,z')}}}$$
where $\ds^{\pm}(z,z') := \min\{\ds(z,z'), \ds(z,-z')		\}$ is the distance from $z$ to either  $z'$ or it's antipodal point $-z'$, whichever is closer (i.e. whichever is in the same hemisphere as $z$).
\end{lemma}

\begin{proof}
By symmetry, we may apply a rotation and assume that $z'=0$ is the north pole, whereby our inner product becomes
$$\left|\langle Z_z^{(s)}, Z_0^{(s)}\rangle\right| $$
Since Legendre polynomials are either even or odd, the zonal spherical harmonics are either even or odd with respect to antipodal reflection, and so since we are taking absolute value of the inner product we may assume that $z$ is in the upper hemisphere; i.e. that 
$$0\leq \ds(z,0) =\ds^{\pm}(z,0) \leq \frac{\pi}{2}$$

Since $Z_0^{(s)}\in\mathcal{H}_s$, the reproducing property (\ref{reproducing}) implies, together with \eqref{e:Zest}, that
$$\langle Z_z^{(s)}, Z_0^{(s)}\rangle = Z_0^{(s)}(z) \lesssim \sqrt{\frac{s}{d_\sph (z,0)}}.$$
\end{proof}

\subsection{Quantum Ergodicity in $\mathcal{H}_s$}\label{s:QE Hs}

We now turn to Theorems~\ref{t:spheremain} and \ref{t:spherealgebraic}.  For each $x \in \mathbb{S}^2$, our set of rotations $g_1,g_2,\ldots g_N$ generates a regular combinatorial tree, embedded (non isometrically) in $\sph$, that we will denote by 
$$\mathfrak{X}(x) = \{g\cdot x : \, \, g\in \,\, \langle g_1, g_2, \ldots, g_N \rangle \}$$
If $x$ and $y$ are two vertices of the same tree (that is, if $\mathfrak{X}(x) = \mathfrak{X}(y)$), we will denote by $d(x,y)$ the geodesic distance on the tree between $x$ and $y$, meaning the length of the shortest path in $\mathfrak{X}(x)$ between the two vertices. 
As above, the geodesic distance on the sphere between two points $z_1$ and $z_2$ will be denoted by $\ds(z_1,z_2)$.

The Laplacian commutes with isometries, and thus with the action of the operator $T_q$ defined in (\ref{Tq}) averaging over this generating set of rotations; thus $T_q$ acts on each Laplace eigenspace $\mathcal{H}_s$. Similarly to Section \ref{qe on graphs} we will say that $T_q$ acting on a subspace of $L^2(\sph)$ has a spectral gap $\beta$ if the spectrum of ${T_q}$ for the action on this subspace is included in
\begin{equation}\label{e:gapsphere}
\left[ - 2\cosh\left(\frac{\log q}2 -\beta \right), 2\cosh\left(\frac{\log q}2 -\beta \right) \right] \cup \left\{ 2\cosh\left(\frac{\log q}2\right) \right\}.
\end{equation}

We fix once and for all an orthonormal basis $\{\psi_j^{(s)}\}_{j=1}^{2s+1}$ of each $\mathcal{H}_s$, consisting of $T_q$-eigenfunctions $T_q\psi_j^{(s)} = \lambda(s,j)\psi_j^{(s)}$.

For any $r \in \N$, we denote by $S_r(x)$ the sphere of radius $r$ in the tree $\mathfrak{X}(x)$, and by $B_r(x)$ the ball of radius $r$.
For $n\in\N$, recall that $P_n$ is the sequence of Chebyshev polynomials of the first kind, defined by the relation
$$ P_n(\cos\theta) = \cos(n\theta). $$

We now fix a test function $a \in L^2(\sph)$. For the proof of Theorem~\ref{t:spheremain} we will take $a$ to be a linear combination of spherical harmonics. The corresponding subspace being of finite dimension, the existence of a spectral gap $\beta_a$ on this subspace is automatic and will depend on $a$. We then extend the result to any continuous test function $a$, by using the fact that $a$ can be uniformly approximated by finite linear combinations of spherical harmonics. For the proof of Theorem~\ref{t:spherealgebraic} we take $a \in L^\infty(\sph)\subset L^2(\sph)$. The additional condition that the rotations have algebraic entries guarantees that $T_q$ will have a spectral gap on all of $L^2(\sph)$ by \cite{BG}. 
 For simplicity we will finally assume in both cases, without loss of generality, that $a \in L^2_0(\sph)$, that is $\int_\sph a \, d\sigma = 0$.  

We identify the test function $a$ with the operator of multiplication by $a$, and we are interested in the following operator acting on functions of the sphere:
\begin{equation}\label{e:timeaverage}
	\frac1{T} \sum_{n=1}^T P_{2n}(T_q/2) a P_{2n}(T_q/2).
\end{equation}
In order to work inside $\mathcal{H}_s$, we use the operator of convolution with $Z_0^{(s)}$. To wit, we write the kernel of $P_{2n}(T_q/2) a P_{2n}(T_q/2)$ as the function $K_{2n}(x,y)$ for $x \in \sph$ and $y \in \mathfrak{X}(x)$, such that any $u \in L^2(\sph)$ satisfies,
$$P_{2n}(T_q/2) a P_{2n}(T_q/2)u(x) = \sum_{y \in \mathfrak{X}(x)} K_{2n}(x,y)u(y).$$
We then define a kernel $[K_{2n}\ast Z^{(s)}](x,y)$ for every $x,y \in \sph$ by
$$ [K_{2n}\ast Z^{(s)}](x,y) = \sum_{z\in \mathfrak{X}(x)} Z^{(s)}_z(y) K_{2n}(x,z), $$
where $Z^{(s)}_z$ is the zonal spherical harmonic of degree $s$ centered at $z\in\mathbb{S}^2$ defined in Section \ref{s:harmonic analysis}.
Note that $K_{2n}\ast Z^{(s)}$ is the kernel of the operator 
$$P_{2n}(T_q/2)aP_{2n}(T_q/2)\tilde{Z}^{(s)}$$
on $\sph$,
where $\tilde{Z}^{(s)}$ is the operator of convolution with the zonal spherical harmonic $Z_0^{(s)}$ of degree $s$, in the sense that
$$[P_{2n}(T_q/2)aP_{2n}(T_q/2)\tilde{Z}^{(s)}u](x) = \int_\sph [K_{2n}\ast Z^{(s)}](x,y)u(y)d\sigma(y)		$$
Since $\tilde{Z}^{(s)}$ acts trivially on $\mathcal{H}_s$ by the reproducing property (\ref{reproducing}), we have equality of the restrictions
\begin{equation}\label{equality on H_s}
P_{2n}(T_q/2)aP_{2n} (T_q/2)\Big|_{\mathcal{H}_s} = P_{2n}(T_q/2)aP_{2n} (T_q/2)\tilde{Z}^{(s)}\Big|_{\mathcal{H}_s} 
\end{equation}
Indeed, $\tilde{Z}^{(s)}$ is nothing more than the orthogonal projection to $\mathcal{H}_s$.
\bigskip

Now, each rotation of $\sph$ fixes two antipodal points, and we consider the set of points $\mathcal{F}_{16T}\subset \sph$ fixed by a word of length $\leq 16T$ in the generators, where $T$ corresponds to the parameter of the same name in \eqref{e:timeaverage}.  For every $x\in \mathcal{F}_{16T}$, we take the neighborhood (in the sphere metric) of radius $s^{-1/4}$, and define the union of these balls to be
$$\mathcal{E}_s(16T) := \bigcup_{x\in \mathcal{F}_{16T}} B_\sph (x, s^{-1/4})$$
These are the points $x$ that are ``very close" (we think of $s$ as being large) to one of their images in $B_{16T}(x)$. Note that $x\notin \mathcal{E}_s(16T)$ implies that $B_{4T}(x)\cap\mathcal{E}_s(8T)=\emptyset$; i.e. any point $g\cdot x\in B_{4T}(x)$ cannot be $s^{-1/4}$-close to a fixed point of a word $h$ of length $\leq {8T}$, since then $x$ would have to be $s^{-1/4}$- close to a fixed point of $g^{-1}hg$.

We will require the following on the parameters $T$ and $s$:
For  any $x\notin\mathcal{E}_s(16T)$, and $z\in B_{4T}(x)$, we have
\begin{equation}\label{e:Tcondition}
\left\{
\begin{array}{ll}
B_\sph(z, s^{-1/2})\cap B_{4T}(x) & =  \{z\} \\
B_\sph(-z, s^{-1/2})\cap B_{4T}(x) & =  \emptyset
\end{array}
\right.
\end{equation}
The free group assumption guarantees that for all $T$ we can find $s_0$ large enough so that the condition is satisfied for any pair $T, s$ with $s \geq s_0$. 
Condition \eqref{e:Tcondition} allows us to apply the methods of Sections~\ref{qe on graphs} and \ref{main estimate graphs} to the kernel $K\ast Z^{(s)}$ on the sphere (see Lemma~\ref{l:HSnorm}).

In the case of algebraic rotations, it is sufficient to take $T < c\log{s}$ for some small constant $c>0$ depending on our generating set of rotations, as is shown in the following lemma.

\begin{lemma}\label{l:algebraicity and discreteness}
There exists a constant $c>0$ depending only on the generating set $\{g_1,\ldots, g_N\}$, such that for all $T < c\log s$, condition \eqref{e:Tcondition} is satisfied for all $x\notin\mathcal{E}_s(16T)$ and $z\in B_{4T}(x)$.
\end{lemma}

\begin{proof}
Let $K$ be a finite extension of $\IQ$ containing all entries of the matrices $g _ 1, \dots, g _ N$.
Recall the notion of (logarithmic) height of an algebraic number $\alpha \in K$ from  e.g.~\cite[Ch.~1]{Bombieri-Gubler}. What is important for us is that it is a nonnegative real number measuring the complexity of nonzero $\alpha \in K$ (e.g. for a rational number given in lowest terms by $p/q$ we have that $h(p/q)= \log \max (p,q)$) with the following properties:
\begin{enumerate}[label=(\alph*)]
\item $h(\alpha \beta) \leq h (\alpha) + h (\beta)$
\item $h (\alpha + \beta) \leq h (\alpha) + h (\beta) + \log 2$ (cf. \cite[\S1.5.16]{Bombieri-Gubler})
\item for any embedding $\iota: K \hookrightarrow \IC$ and any algebraic number $\alpha \neq 0$ we have that $\abs {\iota (\alpha)} \geq e^{-[K:\IQ]h(\alpha)}$ (cf. \cite[\S1.5.19]{Bombieri-Gubler}).
\item $h (\alpha ^{-1}) = h (\alpha)$
\end{enumerate}
It follows that if we set for $g \in \SO (3, \overline{\IQ})$ the height $h (g)$ of $g$ to be the maximum of the heights of its coordinates then for any $g _ 1, g _ 2 \in \SO (3, \overline {\IQ})$
\begin{align*}
h (g _ 1 g _ 2) &\leq h (g _ 1) + h (g _ 2) + O (1) \\
h (g _ 1 ^{-1}) &\leq O (h (g _ 1) + 1)
.\end{align*}
Thus if $w (g _ 1, \dots, g _ N)$ is any word of length $\ell$ in the given generating set
\begin{equation*}
h (w (g _ 1, \dots, g _ N)) \ll \ell (1+\max (h (g_1) ,\dots, h (g _ N ))
.\end{equation*}
Since $g _ 1, \dots, g _ N$ generate a free group, it follows from the above and the basic property (c) of heights that for any reduced word $w$ of length $\ell = 8T$ if $g=w(g_1, \dots g _ N)$ then $\norm {g-1} \geq M^{-8T}$ for some $M$ depending only on the generating set, in other words $g$ is a rotation of angle $\theta$ around its axis with $\abs {\theta} \geq M ^ {-8 T}$.

Choosing $c$ small enough, we can guarantee that 
$$M^{-8T} > M^{-8c\log{s}} = s^{8c\log{M}}> s^{-1/4}$$ 
and thus each word of length $\leq 8T$ in the generators is a rotation through an angle  $\geq s^{-1/4}$.  This means that for any $x\notin \mathcal{E}_s(16T)$, and any pair of distinct points $z,z' \in B_{4T}(x)$ --- which implies $z,z' \notin \mathcal{E}_s(8T)$--- there exists a non-trivial word $g$ of length at most $8T$ in the generators such that $g.z=z'$.  By the above argument, this means that $g$ must be a rotation through an angle of at least $\geq s^{-1/4}$, which for points $z\notin \mathcal{E}_s(8T)$ means that $d_\sph(z,g.z) \gtrsim s^{-1/2}$.  Since this is true for all $z\neq z'\in B_{4T}$, the first part of (\ref{e:Tcondition}) follows after adjusting $c$ to absorb the implied constant.  

For the second part of (\ref{e:Tcondition}), simply observe that if there exists a rotation $g$ such that $d_\sph(-z, g.z)$ is small, then $d_\sph(z, g^2.z) = 2d_\sph(-z,g.z)$; and so by further reducing the constant $c$, we may incorporate the second part of (\ref{e:Tcondition}) as well.
\end{proof}

\medskip

We can now state the following central estimate. It will be proved in Section~\ref{main estimate sphere}.

\begin{proposition}\label{p:egorov}
Let $T,s > 0$ satisfying condition \eqref{e:Tcondition}. Let $a \in H \subset L^2(\sph)$, where $H$ is a subspace closed under the action of $T_q$, and such that ${T_q}{\restriction_{H}}$ has a spectral gap $\beta_H$, as defined in \eqref{e:gapsphere}. Then
 $$\left\| \frac1T \sum_{n=1}^T P_{2n}(T_q/2) a P_{2n}(T_q/2)\tilde{Z}^{(s)} \right\|_{HS}^2 
 \lesssim \frac{s}{\beta_H^2 T} \|a\|_2^2 + s^{1/2} q^{16T} \|a\|_\infty^2,$$
 with an implied constant depending only on $q$.
\end{proposition}

\vspace{.1in}

We will now prove Theorems \ref{t:spheremain} and \ref{t:spherealgebraic} using this estimate. We have the general lemma
\begin{lemma}\label{l:HSQE}
For any function $a \in L^2(\sph)$,
	$$ \sum_{j=1}^{2s+1} |\la \psi^{(s)}_j, a \, \psi^{(s)}_j \ra|^2 
	\lesssim \left\| \frac1T \sum_{n=0}^T P_{2n} a P_{2n}\tilde{Z}^{(s)} \right\|_{HS}^2, $$
where the implied constant is absolute.
\end{lemma}
\begin{proof}
By Lemma~\ref{l:unitarity} and the equality of operators (\ref{equality on H_s}) on $\mathcal{H}_s$, we have as in the end of Section~\ref{qe on graphs}
\begin{align*}
   \sum_{j=1}^{2s+1} |\la \psi^{(s)}_j, a \psi^{(s)}_j \ra|^2 
  &\lesssim  \sum_{j=1}^{2s+1} \left|\frac{1}{T} \sum_{k=1}^T \cos^2 (2n\theta)\right|^2|\la \psi^{(s)}_j, a \psi^{(s)}_j \ra|^2 \\
  &\lesssim  \sum_{j=1}^{2s+1} \left|\left\la \psi^{(s)}_j, \frac1T \sum_{n=0}^T P_{2n} a P_{2n} \psi^{(s)}_j \right\ra\right|^2 \\
    &\lesssim  \sum_{j=1}^{2s+1} \left|\left\la \psi^{(s)}_j, \frac1T \sum_{n=0}^T P_{2n} a P_{2n} \tilde{Z}^{(s)}\psi^{(s)}_j \right\ra\right|^2 \\
  &\lesssim \left\| \frac1T \sum_{n=0}^T P_{2n} a P_{2n}\tilde{Z}^{(s)} \right\|_{HS}^2
\end{align*}
\end{proof}

\begin{proof}[Proof of Theorem \ref{t:spheremain}]
Let $k \in \N$, $s_1,\ldots,s_k \in \N \setminus \{0\}$, and 
$$ H = \bigoplus_{i=1}^k \mathcal H_{s_i}. $$
Notice  that $H$ is closed under the action of $T_q$ and that ${T_q}{\restriction_{H}}$ has no eigenvalues equal to $\pm 2\cosh(\log q / 2)$. The space $H$ being finite dimensional, ${T_q}{\restriction_{H}}$ has a finite number of eigenvalues and thus has trivially a spectral gap $\beta_H$ as defined in Section \ref{s:QE Hs}. We first prove the theorem for a test function $a \in H \subset L^\infty(\sph)$ (as $H$ is orthogonal to $\mathcal H_{0}$ we have $\int_\sph a \, d\sigma = 0$). 
Combining Lemma~\ref{l:HSQE} and Proposition~\ref{p:egorov} we get
\begin{align*}
  \frac1{2s+1}\sum_{j=1}^{2s+1} |\la \psi^{(s)}_j, a \psi^{(s)}_j \ra|^2
  &\lesssim \frac{1}{\beta_H^2 T} \|a\|_2^2 + s^{-1/2} q^{16T} \|a\|_\infty^2.
\end{align*}
We fix $T$ and use the free group assumption to find $s_0$ such that \eqref{e:Tcondition} is satisfied for the pair $T, s$ with $s \geq s_0$. We have
$$  \lim_{s \to +\infty} \frac1{2s+1} \sum_{j=1}^{2s+1} |\la \psi^{(s)}_j, a \psi^{(s)}_j \ra|^2 \lesssim  \frac{1}{\beta_H^2 T} \|a\|_2^2. $$
We then use the fact that we can find a sequence of test functions $a_n$ such that $$ \|a_n - a \|_\infty \to 0,$$ by density of finite linear combinations of spherical harmonics in the space of continuous function for the uniform convergence, to extend the conclusion to the case of a general continuous test function.
\end{proof}

\begin{proof}[Proof of Theorem \ref{t:spherealgebraic}]
In this case we assume that the rotations are algebraic. We can thus take a general test function $a \in L^2(\sph)$, then \cite{BG} tells us $T_q$ has a spectral gap as an operator acting on this space. We assume in addition that $a\in L^\infty(\sph)$ and $\int_\sph a \, d\sigma = 0$.

Combining Lemma~\ref{l:HSQE} and Proposition~\ref{p:egorov} we get
\begin{align*}
  \frac1{2s+1} \sum_{j=1}^{2s+1} |\la \psi^{(s)}_j, a \psi^{(s)}_j \ra|^2
  &\lesssim \frac{1}{\beta^2 T} \|a\|_2^2 + s^{-1/2} q^{16T} \|a\|_\infty^2.
\end{align*}
To finish the proof, we take $T = c_0 \log s$ with $c_0 < c $ (where the constant $c$ is given by Lemma \ref{l:algebraicity and discreteness}), so that condition \eqref{e:Tcondition} is satisfied, and $s^{-1/2} q^{16T} < s^{-\epsilon}$ for some $\epsilon > 0$.
\end{proof}

\section{Proof of the main estimate on the sphere}\label{main estimate sphere}

Here we prove Proposition~\ref{p:egorov}, assuming the rotations are algebraic. Recall that this assumption is used to get both a spectral gap for $T_q$ on very general spaces of test functions, and a quantitative relationship between the parameters $s$ and $T$. We use it in this section only for the quantitative aspect. The same ideas apply for general rotations generating a free group--- under the assumption assumption that $a$ is a finite linear combination of spherical harmonics--- except that we would need to replace all quantitative relationships between $s$ and $T$ with a more qualitative ``$s$ large enough with respect to $T$''.

The next Lemma enables us to apply the techniques of Section~\ref{main estimate graphs} to our kernels on $\sph$, outside of the small set $\mathcal{E}_s(16T)$ of ``bad" points:
\begin{lemma}\label{l:HSnorm}
For any $T,s$ satisfying condition \eqref{e:Tcondition},
and any function $K(x,y)$ defined on $\{ (x,y) \in \sph \times \sph \, | \, y \in \mathfrak X(x) \}$ such that $K(x,y)=0$ whenever $d(x,y)>{4T}$, we have
 \begin{align*}
\int_{\sph} |[K\ast Z^{(s)}](x,y)|^2 \, d\sigma(y) &\lesssim s \sum_{z\in\mathfrak{X}(x)} |K(x,z)|^2\cdot  \left\{  \begin{array}{ll} 1 & x\notin \mathcal{E}_s(16T)\\ q^{4T} & x\in\mathcal{E}_s(16T) \end{array} \right. 
 \end{align*}
 \end{lemma}

\begin{proof}
 \begin{align*}
|[K \ast Z^{(s)}](x,y)|^2 
  & =  \sum_{z\in \mathfrak{X}(x)}Z_z^{(s)}(y)^2 | K(x,z) |^2 \\
  &  \quad +  \sum_{z,z' \in \mathfrak{X}(x), z \neq z'} Z_z^{(s)}(y)Z_{z'}^{(s)}(y) \overline{K(x,z)} K(x,z')
   \end{align*}
 
 The first term in the sum is easily handled; the integral over $y$ is applied only to $Z_z^{(s)}(y)^2$, and we know that the $L^2$-norm of $Z_z^{(s)}$ is given by 
 $$||Z_z^{(s)}||^2_2 = \frac{2s+1}{4\pi} \lesssim s$$
 so that
 $$\int_{\sph} \sum_{z\in \mathfrak{X}(x)}Z_z^{(s)}(y)^2 | K(x,z) |^2 d\sigma(y) \lesssim s\cdot \sum_{z\in\mathfrak{X}(x)} |K(x,z)|^2$$
 
We turn our attention to the second sum, and consider first the case where $x\notin\mathcal{E}_s(16T)$. By Condition \eqref{e:Tcondition}, for any point $z\in\mathfrak{X}(x)$ such that $d(x,z) \leq {4T}$ on the tree, the balls $B_{\mathbb{S}^2}(\pm z,s^{-1/2})$  do not contain any other point $z'\in\mathfrak{X}(x)$ such that $d(x,z') \leq {4T}$. Then by Lemma~\ref{inner product z z'}, the integral over $y$ gives 
$$\langle Z_z^{(s)}, Z_{z'}^{(s)}\rangle \lesssim s^{3/4}$$ 
for all terms in the second sum, and so the contribution to the integral (over $y$) of the second sum is estimated by
\begin{eqnarray*}
 \sum_{z,z'\in \mathfrak{X}(x), z\neq z'}  \langle Z_z^{(s)}, Z_{z'}^{(s)}\rangle\overline{K(x,z)}K(x,z')
& \lesssim & s^{3/4} \sum_{z,z'\in \mathfrak{X}(x)} \overline{K(x,z)}K(x,z')\\
& \lesssim & s^{3/4} \left|\sum_{z\in\mathfrak{X}(x), d(x,z)\leq {4T}} K(x,z)\right|^2\\
& \lesssim & s^{3/4}q^{4T} \sum_{z\in\mathfrak{X}(x)} |K(x,z)|^2\\
& \lesssim & s\cdot \sum_{z\in\mathfrak{X}(x)} |K(x,z)|^2
\end{eqnarray*}
if $c$ is sufficiently small so that $q^{4T} < s^{1/4}$.

If $x\in\mathcal{E}_s(16T)$, then we  estimate the inner product trivially by
$$\langle Z_z^{(s)}, Z_{z'}^{(s)}\rangle \leq ||Z_z^{(s)}||_2||Z_{z'}^{(s)}||_2 \lesssim s$$
and apply the same calculation to get
$$\int_{\sph} |[K\ast Z^{(s)}](x,y)|^2 \, d\sigma(y) \lesssim sq^{4T} \sum_{y\in\mathfrak{X}(x)} |K(x,y)|^2$$
\end{proof}

We now proceed as in the proof of Proposition~\ref{p:egorov_graphs} in Section~\ref{main estimate graphs}.
In what follows, we will denote simply by $P_n$ the operator $P_n(T_q/2)$.

The rest of the argument closely follows the proof on graphs from Section~\ref{main estimate graphs}. We denote by $K_T(x,y)$ the kernel of the operator $$A_T = \frac1T \sum_{n=0}^T P_{2n} a P_{2n} \tilde Z^{(s)}.$$ Let us first give an explicit expression for the kernel $K_{2n}$ of the operator $P_{2n}aP_{2n}$. For this purpose we define the sets
 \[E_{j,k} = E_{j,k}(x,y) = \{ z \in \mathfrak{X}(x): d(x,z) = j, d(y,z) = k \}.\] 
The kernel $K_{2n}(x,y)$ is then given by
\begin{align}\label{e:sphere_kernel}
 K_{2n}(x,y) &= \frac1{4q^{2n}} \sum_{z\in E_{2n,2n}} a(z) \\
 &\quad + \frac{(1-q)}{4q^{2n}} \sum_{j=0}^{n-1}\left( \sum_{z\in E_{2j,2n}} a(z) + \sum_{z\in E_{2n,2j}} a(z) \right) \nonumber \\
 &\quad + \frac{(1-q)^2}{4q^{2n}} \sum_{j,k=0}^{n -1} \sum_{z\in E_{2j,2k}} a(z). \nonumber
\end{align}
This kernel is equal to $0$ whenever $d(x,y)$ is odd or $d(x,y) > 4n$. We then have $K_T = \frac1T \sum_{n=0}^T [K_{2n}\ast Z^{(s)}](x,y)$.

We now separate the good and bad points. Define
\begin{equation*}
A_T' f(x) = \left\{
\begin{array}{ll}
   A_T f(x) & \text{if } x \notin \mathcal E_s(16T) \\
   0 & \text{otherwise.}
\end{array}
\right.
\end{equation*}
We then have
\begin{lemma}
$$ \|A_T\|_{HS}^2 \leq \| A_T' \|_{HS}^2 + O\left(s^{1/2} q^{16T} \|a\|_\infty^2\right) $$
\end{lemma}

\begin{proof}
We have
$$
\| A_T - A'_T\|_{HS}^2 = \frac{1}{T^2} \int_{x\in\mathcal{E}_s(16T)}\int_{y\in\sph} \left|\sum_{n=0}^T[K_{2n}\ast Z^{(s)}](x,y)\right| ^2 d\sigma(y)d\sigma(x)
$$

We then use Lemma~\ref{l:HSnorm} to write
\begin{align*}
\| A_T - A'_T\|_{HS}^2 
	&\lesssim \frac{s q^{4T}}{T^2} \int_{x\in\mathcal{E}_s(16T)}  \sum_{\substack{y \in \mathfrak{X}(x)\\ d(x,y) \leq 4T}} \left| \sum_{n=0}^T K_{2n}(x,y)\right|^2 \, d\sigma(x) \\
	&\lesssim s q^{8T} \sigma(\mathcal{E}_s(16T)) \|a\|_\infty^2,
\end{align*}
where we used the fact that $\sup_{x,y} K_{2n}(x,y) \leq \| a \|_\infty$. We then notice that $\mathcal{E}_s(16T)$ is a union of $O(q^{16T})$ balls of radius $s^{-1/4}$ and obtain
$$ \| A_T - A'_T\|_{HS}^2 \lesssim s^{1/2} q^{16T} \|a\|_\infty^2.$$
\end{proof}

We can now restrict our attention to the operator $A_T'$. We write 
$$ A_T' = \frac1T \sum_{n=0}^T \frac1{q^{2n}} \sum_{l=0}^{2n}\sum_{j,k=0}^n c(j,k) S_{2j,2k,2l}, $$
where the operator $S_{2j,2k,2l}$ is defined by its kernel on points $x,y \in \sph$
\begin{equation}\label{e:S kernel sphere}
[S_{2j,2k,2l}](x,y) = \sum_{z\in\mathfrak X(x)} Z_z^{(s)}(y) K(x,z),
\end{equation}
and $K(x,y)$ is given by
 $$K(x,y) = \mathbf{1}_{\{x\notin \mathcal E_s(16T), \; y\in \mathfrak X(x), \; d(x,y)=2l \}} \sum_{z\in E_{2j,2k}(x,y)} a(z). $$

Interchanging the sums in $l$ and $n$, which will be useful later, we obtain
\begin{equation}\label{e:AT sphere}
 A_T' = \frac1T \sum_{l=0}^{2T} \sum_{n=\lceil l/2 \rceil}^{T} \frac1{q^{2n}} \sum_{j,k=0}^n c(j,k) S_{2j,2k,2l},
\end{equation}

The proof of Proposition \ref{p:egorov} will follow from the following estimate:
 \begin{lemma}\label{l:Sop sphere}
$$ \| S_{2j,2k,2l} \|_{HS} \lesssim s^{1/2} \, q^{(k+j)} e^{-\frac{\beta}2 (k+j-l)} \|a \|_2 $$
 \end{lemma}
 Indeed, the estimation of the Hilbert-Schmidt norm of $A'_T$ using this lemma is here exactly the same as in the case of graphs up to a factor $s$. We reproduce it for the convenience of the reader. Let us first notice that for $l\neq l'$, the operators $S_{2j,2k,2l}$ and $S_{2j',2k',2l'}$ are orthogonal with respect to the Hilbert-Schmidt norm. We deduce from this fact and the expression \eqref{e:AT sphere} that
 \begin{align*}
\|A_T' \|^2_{HS} &= \frac1{T^2} \sum_{l=0}^{2T} \left\| \sum_{n=\lceil l/2 \rceil}^T \frac1{q^{2n}} \sum_{j,k=0}^n   c(j,k) S_{2j,2k,2l} \right\|_{HS}^2 \\
&\lesssim \frac1{T^2} \sum_{l=0}^{2T} \left( \sum_{n=\lceil l/2 \rceil}^T \frac1{q^{2n}} \sum_{j,k=0}^n \| S_{2j,2k,2l} \|_{HS}  \right)^2 \\
&\lesssim \frac{s}{T^2} \sum_{l=0}^{2T} \left( \sum_{n=\lceil l/2 \rceil}^T \frac1{q^{2n}} \sum_{j,k=0}^n q^{(k+j)} e^{-\frac{\beta}2 (k+j-l)} \|a\|_2 \right)^2 \\
&\lesssim \frac{s}{T^2} \sum_{l=0}^{2T}  \left( \sum_{n=\lceil l/2 \rceil}^T e^{-\beta(n-l/2)} \|a\|_2 \right)^2 \\
&\lesssim \frac{ s \|a\|^2_2}{\beta^2 T}.
\end{align*}

We now want to prove Lemma~\ref{l:Sop sphere}. We will first write the kernel of the operators $S_{2j,2k,2l}$ in a more convenient way.
We consider the space $L^2(\sph')$, where $\sph' = \{(x,y) \in \sph \times \sph, \exists i \in \{1,\ldots,N\}, g_i\cdot x = y \}$ and the measure is given by $\sum_{i=1}^N d\sigma(x) \otimes \delta_{g_i \cdot x}$. We define the operator 
$$ T_q'F(x,y) = \frac1{q} \sum_{g_i \cdot y \neq x} F(y,g_i \cdot y). $$
We then have the following analog of Lemma \ref{l:edgegap}.
\begin{lemma}\label{l:edgegapsphere}
For every $k \geq 1$, we have
$$ \| (T_q')^k \| \lesssim e^{-\beta k}$$
with an implicit constant depending only on $q$.
\end{lemma}

\begin{proof}
The maps $B,E : L^2(\sph) \to L^2(\sph')$  are here given by
$$ Bf(x,y) = f(x), Ef(x,y) = f(y) $$
The space $L^2(\sph')$ can be divided into $\text{Im}(B)\oplus\text{Im}(E)$ and its orthogonal complement given by the functions $F$ such that for any $x \in \sph$
$$ \sum_{i=1}^N F(x,g_i x) = \sum_{i=1}^N F(g_i x,x) = 0$$
On this orthogonal complement the action of $T_q'$ is given by
$$ T_q'F(x,y) = - \frac{1}{q} F(y,x), $$
and we can decompose $\text{Im}(B)\oplus\text{Im}(E)$ using an orthonormal basis of $T_q$. For every eigenfunction $w \in L^2(\sph)$ with $T_q$-eigenvalue $\lambda$, the space $\C(Bw) + \C(Ew)$ is stable under $T_q'$ and the action of $T_q'$ is given by the matrix
\begin{equation*}
 \begin{pmatrix}
  0 & -\frac{1}{q} \\
  1 & \frac{\lambda}{\sqrt{q}}
 \end{pmatrix}
\end{equation*}
The rest of the proof is completely analogous to the proof of Lemma \ref{l:edgegap}.
\end{proof}

The following analog of Lemma~\ref{l:Tq sum} allows us to rewrite the kernel of $S_{2j,2k,2l}$ using the operator $T'_q$.

 \begin{lemma}\label{l:Tq sum sphere}
 Let $0 \leq j,k \leq n$.
 Let $x \in \sph$ and $y \in \mathfrak{X}(x)$ such that $l = d(x,y)/2$ satisfies $|k-j| \leq l \leq k+j$. Denote by $w$ the vertex of the segment $[x,y]$ such that  $d(x,w) = l - (k-j)$ and $d(y,w) = l + (k-j)$ in the case $k\geq j$, and $d(x,w) = l + (k-j)$ and $d(y,w) = l - (k-j)$ in the case $j \geq k$. Then
  \begin{equation}\label{e:vertextoedge sphere}
  \sum_{z\in E_{2j,2k}} a(z) = \sideset{}{'}\sum_{g_i} q^{k+j-l} (T'_q)^{k+j-l} B a (w,g_i w)
  \end{equation}
  where the sum on the right-hand side runs only over the rotations $g_i$ such that $g_i w$ is not on the segment $[x,y]$.
 \end{lemma}
We will also use the following lemma, a straightforward analog of Lemma~\ref{l:edgenorm}.
\begin{lemma}\label{l:L2edge}
 Let $k,l$ be fixed integers such that $k \leq 2l \leq 4T$. For any $x \in \mathbb{S}^2$ and $y \in \mathfrak X(x)$ such that $d(x,y) = 2l$, let $w$ be a vertex of the segment $[x,y]$ such that $d(x,w) = k$.  Then
 $$ \int_\sph \sum_{y\in S_{2l}(x)} \sum_{g} |F (w,gw)| d\sigma(x) \leq q^{2l} \int_\sph \sum_{g} | F (x,gx)| d\sigma(x) $$
\end{lemma}

\begin{proof}[Proof of Lemma \ref{l:Sop sphere}]
We first use Lemma~\ref{l:HSnorm} together with \eqref{e:S kernel sphere} to write
$$ \|  S_{2j,2k,2l} \|_{HS}^2 \lesssim s \int_{x\notin\mathcal E_s(16T)} \sum_{\substack{y\in \mathfrak X(x)\\ d(x,y)=2l}} \left| \sum_{z\in E_{2j,2k}(x,y)} a(z) \right|^2 d\sigma(x)$$
We then use Lemma \ref{l:Tq sum sphere} and Lemma \ref{l:L2edge} to obtain
 \begin{align*}
 \|  S_{2j,2k,2l} \|_{HS}^2 & \lesssim s \int_{x\notin\mathcal E_s(16T)} \sum_{\substack{y\in \mathfrak X\\ d(x,y)=2l}} \left| \sideset{}{'}\sum_{g_i} q^{k+j-l} (T'_q)^{k+j-l} B a (w,g_i w) \right|^2 d\sigma(x) \\
 &\lesssim s \, q^{2(k+j-l)} q^{2l} \int_\sph \sum_g \left| (T'_q)^{k+j-l} B a (x,gx) \right|^2 d\sigma(x).
 \end{align*}
 Finally, Lemma \ref{l:edgegapsphere} gives
 $$ \| S_{2j,2k,2l} \|_{HS}^2 \lesssim s \, q^{2(k+j)} e^{-\beta(k+j-l)} \|a\|_2^2.$$
\end{proof}

\section{Kesten-McKay Law for Spherical Harmonics}
\label{s:Kesten-McKay}

In this section, we prove the following Kesten-McKay law for the space $\mathcal{H}_s$ of spherical harmonics:
\begin{lemma}\label{sphere kesten mckay}
Let $\{\psi_j^{(s)}\}_{j=1}^{2s+1}$ be an orthonormal basis of $\mathcal{H}_s$ consisting of $T_q$-eigenfunctions with $T_q\psi_j^{(s)} = \lambda(s,j) \psi_j^{(s)}$, and let $I\subset [-2,2]$ be an arbitrary interval in the tempered spectrum of $T_q$.  Set
$$N(I,s) := \#\{j : \lambda(s,j) \in I\}$$
to be the dimension of the subspace of $\mathcal{H}_s$ spanned by those $T_q$-eigenfunctions with $T_q$-eigenvalue in $I$.  Then
$$N(I,s)\sim C_q(I) \cdot s$$
as $s\to\infty$, where $C_q(I)>0$ depends only on the interval $I$, and the size of the generating set of rotations.
\end{lemma}

\begin{proof}

This is a standard argument that we bring here for completeness.  The idea is to estimate the moments 
$$M_n(s) = \sum_{j=1}^{2s+1} \lambda(s,j)^n = Tr\left(T_q^n\restriction_{\mathcal{H}_s}\right)$$
and show that they agree asymptotically with the moments of the Plancherel measure for the $q+1$-regular tree; solving the inverse moment problem then implies that the eigenvalues of $T_q$ in $\mathcal{H}_s$ are asymptotically distributed according to a continuous, positive distribution on $(-2,2)$; namely, the $q$-adic Plancherel measure.  In particular, this implies that any fixed interval receives a positive proportion of the $T_q$-eigenvalues as $s\to\infty$.

We begin again with the zonal spherical harmonic $Z_z^{(s)}\in\mathcal{H}_s$, which can be written in terms of the $\{\psi_j^{(s)}\}$--- indeed, in terms of any orthonormal basis of $\mathcal{H}_s$--- as
$$Z_z^{(s)}(y) = \sum_{j=1}^{2s+1} \psi_j^{(s)}(z)\overline{\psi_j^{(s)}(y)} $$  
Since the $\{\psi_j^{(s)}\}$ are eigenfunctions of $T_q$, we can write
$$T_q^n Z_z^{(s)}(y) = \sum_{j=1}^{2s+1} \lambda(s,j)^n \psi_j^{(s)}(z) \overline{\psi_j^{(s)}(y)}$$
and so the trace of $T_q^n$ on $\mathcal{H}_s$ is obtained by integrating on the diagonal
\begin{equation}\label{Tq trace}
\int_{z\in\sph} T_q^n Z_z^{(s)}(z) d\sigma(z)= \sum_{j=1}^{2s+1} \lambda(s,j)^n
\end{equation}
We wish to compute the asymptotics of these moments as $s\to\infty$; we will do this by estimating the integrand on the left-hand side point wise.

It is easy to see that
$$T_q^n Z_z^{(s)}(z) = p^{-n/2}N_0(n)Z_z^{(s)}(z) + \sum_{0<d_{\mathfrak{X}}(z,y)\leq n} N_y(n)Z_z^{(s)}(y)$$
where $N_0(n)$ is the number of paths of length $n$ in the $q+1$-regular tree starting and ending at $0$, and similarly $N_y(n)$ is the number of paths of length $n$ from $z$ to $y$.  The number of terms in the second sum, as well as the values of $N_y(n)$, depend only on $n$ and are independent of $s$.  

Now, since the rotations generate a free group, there exists a small exceptional set $z\in\mathcal{E}_s(3n)$ as in Section~\ref{main estimate sphere}, outside of which we know that all points $y$ in the sum are separated from each other and from $z$; i.e. that for $s$ large enough, the distance $d_\sph(z,y)>s^{-1/4}$ for all terms in the second sum.  The estimate (\ref{legendre decay}) then implies that for all $z\notin\mathcal{E}_s(3n)$ we have
$$Z_z^{(s)}(y)\lesssim q^n \max\{N_y(n)\} s^{3/4} \lesssim_n s^{3/4}$$
since the points $y$ and corresponding values of $N_y(n)$ depend only on $n$.
On the other hand 
$$Z_z^{(s)}(z) = Z_0^{(s)}(0) \gtrsim s$$
so that as $s\to\infty$ we have
$$T_q^nZ_z^{(s)}(z) \sim p^{-n/2}N_0(n)Z_0^{(s)}(0)$$
uniformly in $z\in\sph\backslash\mathcal{E}_s(3n)$, and since the right-hand side is independent of $z$, we have
\begin{eqnarray*}
\int_{\sph\backslash\mathcal{E}_s(3n)} T^n_q Z_z^{(s)}(z)d\sigma(z) & \sim & p^{-n/2}N_0(n)Vol(\sph\backslash\mathcal{E}_s(3n))\cdot s\\
& \sim & p^{-n/2}N_0(n)Vol(\sph)\cdot s
\end{eqnarray*}
As in Section~\ref{main estimate sphere}, the integral of $T^n_qZ_z^{(s)}(z)$ over the  set $z\in \mathcal{E}_s(3n)$ is negligible once $s$ is large enough (and thus the set $\mathcal{E}_s(3n)$ small enough).  Therefore the trace satisfies
$$\sum_{j=1}^{2s+1} \lambda(s,j)^n \sim p^{-n/2}N_0(n)Vol(\sph) \cdot s$$
Dividing by $s$ and solving the inverse moment problem (just as in, eg. \cite{SarnakHeckeStatistics}) shows that the asymptotic distribution of the $T_q$-eigenvalues in $\mathcal{H}_s$ is asymptotically proportional to $s$ times the $q$-adic Plancherel measure
$$d\mu_q(x) = \left\{	\begin{array}{ccc}  \frac{	(q+1)\sqrt{4-x^2}}{2\pi (q+q^{-1}+2)-x^2} & \quad\quad\quad & |x|\leq 2\\ 
0 & \quad\quad\quad & |x|\geq 2	\end{array}		\right.$$
which is a continuous and non-negative measure supported in $[-2,2]$, and strictly positive in the interior $(-2,2)$ (note that it vanishes only at the endpoints).  This shows that any interval $I\subset [-2,2]$ contains a number of $T_q$-eigenvalues 
\begin{eqnarray*}
N(I,s) & \sim & s\cdot \int_{x\in I} d\mu_q(x)
\end{eqnarray*}
as required.
\end{proof}

\bibliographystyle{plain}

\def\cprime{$'$}

\end{document}